\documentclass[12pt]{amsart}
\usepackage{amscd,amssymb,verbatim}

\usepackage[all]{xy}
\usepackage{amssymb,amsmath,amsthm,epsfig}

\newtheorem{thm}{Theorem}[section]
\newtheorem{lem}[thm]{Lemma}
\newtheorem{cor}[thm]{Corollary}
\newtheorem{prop}[thm]{Proposition}
\theoremstyle{definition}
\newtheorem{defn}[thm]{Definition}

\newtheorem{remark}[thm]{Remark}

\newcommand{\N}{\mathbb{N}}

\newcommand{\C}{\mathbb{C}}

\newcommand{\R}{\mathbb{R}}
\newcommand{\E}{\mathbb{E}}
\renewcommand{\P}{\mathbb P}

\newcommand{\cJ}{\mathcal J}

\newcommand{\cK}{\mathcal K}

\newcommand{\vp}{\varepsilon}

\newcommand{\ra}{\rangle}
\newcommand{\la}{\langle}

\newcommand{\kin}{\!\in\!}

\newcommand{\FSS}{\mathcal{F}\mathcal{S}\mathcal{S}}
\newcommand{\FD}{\mathcal{F}\mathcal{D}}
\newcommand{\StSi}{\mathcal{S}t\mathcal{S}i}
\newcommand{\sign}{\text{\rm sign}}
\newcommand{\spa}{\text{\rm span}}

\begin{document}

\title{On the  Closed  Subideals of $L(\ell_p\oplus \ell_q)$}
\author{Th.~Schlumprecht}
\thanks{Research partially supported by NSF grant DMS 0856148}
\address{Department of Mathematics, Texas A\&M University, College Station, Texas 77843, USA}
\email{schlump@math.tamu.edu}
\begin{abstract}  In this paper we  first review the  known  results about the closed 
 subideals of the space of bounded operator on $\ell_p\oplus \ell_q$, $1<p<q<\infty$, and  then
 construct several new ones.
\end{abstract}

\maketitle
\section{Introduction}\label{S:0}

For very few Banach spaces $X$ all the closed subideals of $L(X)$, the algebra of all bounded 
and linear operators on $X$, are determined.
In 1941 Calkin \cite{Ca}
 showed that the only proper, non-trivial and   closed
ideal of $L(\ell_2)$ is the ideal of compact operators. The same was
shown to be true for $\ell_p$ $(1\le p<\infty)$ and $c_0$
in~\cite{GMF}.  Until very recently it was open if there are any  other  infinite dimensional Banach
spaces $X$,  for  which the compact operators are the   only  proper,  non-trivial  and closed subideal of $L(X)$.  We call such spaces {\em simple}. Then Argyros and Haydon \cite{AH} established
the  existence  of Banach spaces with a basis on which all operators are a compact perturbation 
of a scalar multiple of the identity. It follows immediately that  such spaces are simple. But it is not known whether or not there are any other simple  spaces admitting an unconditional 
 basis (and thus having a rich structure of operators on them).

The structure of the closed ideals of operators on
non separable Hilbert spaces was independently obtained by Gramsch
\cite{Gr} and Luft \cite{Lu}. Recently Daws \cite{Da}
extended their results to non separable $\ell_p$-spaces, $1\le
p<\infty$, and non separable $c_0$-spaces.

Beyond these spaces the complete
structure of closed ideals in $L(X)$ was  described
in~\cite{LLR} for
$X=\bigl(\bigoplus_{n=1}^\infty\ell_2(n)\bigr)_{c_0}$ and
in~\cite{LSZ} for
$X=\bigl(\bigoplus_{n=1}^\infty\ell_2(n)\bigr)_{\ell_1}$. In  both
cases, there are exactly two nested proper non-zero closed ideals, namely the compacts and the closure of all operators factoring through $c_0$, or $\ell_1$, respectively.  
Apart from those mentioned above, there are no other separable Banach
spaces $X$ for which the structure of the closed ideals in $L(X)$ is
completely known. 
 It is still open whether or not the  closed subideals of the operators on the  spaces $(\oplus_{n=1}^\infty\ell_1(n))_{c_0}$ and $(\oplus_{n=1}^\infty\ell_\infty(n))_{\ell_1}$
  admit the same sublattice structure (for partial results see \cite{LOSZ}). 
An interesting space for studying the closed subideals of its bounded linear operators is the 
space  $X$ introduced in \cite{S}.  This space is {\em  complementably  minimal} \cite{AS1}, which means that every infinite dimensional
 closed subspace of $X$ contains a further subspace which is complemented in $X$ and isomorphic to $X$. This implies that
 the  strictly singular operators  (see the definition at the end of this section) is the only maximal proper closed subideal of $L(X)$.
 As shown in \cite{AS2}, $X$ admits strictly singular but not compact operators, and it is conjectured that $L(X)$ contains infinitely
 many closed subideals, all of which have to lie between the ideal of compact operators, and the ideal of strictly singular operators.  

A space whose closed ideals  of operators attracted the  attention of several researches is the  $p$th quasi reflexive James $J_p$,  with $1<p<\infty$.
  Edelstein and  Mityagin \cite{EM} observed  that the ideal of weakly compact operators on $J_p$ is a  maximal proper subideal of $L(J_p)$
  and  Laustsen proved in   \cite{La} that it is the only one. 
  In  \cite{LW}, for $p=2$, and  in \cite{La}, for general $p\in(1,\infty)$, it was shown that the closure of the operators on $J_p$ factoring through $\ell_2$ 
 contains strictly the  ideal of compact operators and is strictly contained in the ideal of weakly compact operators. Very 
recently Bird, Jameson and Laustsen \cite{BJL} found a new closed sub ideal of $L(J_p)$ and  proved that the closure of the ideal of operators factoring through the $\ell_p$-sum of
$\ell_\infty(n)$, $n\kin\N$, is strictly larger than the closure of the ideal of operators factoring through $\ell_p$ and strictly smaller then the ideal of weakly compact operators.

Although studied in several papers (cf.\cite{Mi},\cite{Pi} and \cite{SSTT}) the structure of
 the  closed subideals of $L(\ell_p\oplus\ell_q)$, $1<p<q<\infty$ remains a mystery. 
 It is not even known whether or not  $L(\ell_p\oplus\ell_q)$, contains infinitely many subideals.
There were several results 
proved in the 1970's concerning various special ideals or special
cases of $p$ and~$q$. We refer the reader to the book by Pietsch
\cite[Chapter~5]{Pi} for details. In particular,
\cite[Theorem~5.3.2]{Pi} asserts that $L(\ell_p\oplus \ell_q)$, with
$1\le p<q$, has exactly two proper maximal ideals (namely, the ideal of
operators which factor through $\ell_p$ and the ideals of operators
which factor through $\ell_q$), and establishes a one-to-one
correspondence between the non-maximal  proper subideals of  $L(\ell_p\oplus \ell_q)$
and the closed  ideals in $L(\ell_p,\ell_q)$. 
 By proving that the formal identity $I(p,q):\ell_p\to\ell_q$ is {\em finitely strictly singular} (see the definition at the end of this section)
  and establishing the existence of an operator $T:\ell_p\to \ell_q$ which is not  finitely strictly singular Milman \cite{Mi} 
  concluded that  $L(\ell_p,\ell_q)$ contains at least two non trivial, proper and closed subideals. 
In \cite{SSTT} the study of the structure of the closed subideals of $L(\ell_p,\ell_q)$ 
was continued, and, among other results, it was discovered that the  lattice of subideals
of $L(\ell_p,\ell_q)$ 
is not linearly ordered,   and contains at least 4  nontrivial, proper and closed subideals if $1<p<2<q<\infty$. 
In this paper we  increase this number to 7.

In Section \ref{S:2} we will recall the known results on the closed subideals of $L(\ell_p\oplus\ell_q)$
and $L(\ell_p,\ell_q)$, and sketch the proof of several of them. In Section \ref{S:3} we will 
formulate and prove our main result (see Theorem \ref{T:2.1}).

Let us first recall some necessary notation.

If $X$ and $Y$ are Banach spaces, $L(X,Y)$
denotes the space of bounded linear operators $T:X\to Y$, and if $X=Y$ we write $L(X)$ instead
of  $L(X,X)$.
A linear subspace $\cJ\subset L(X,Y)$, is called a {\em subideal of $L(X,Y)$}, if for all $A\in L(Y)$,  
$B\in L(X)$, and $T\in \cJ$ also $A\circ T\circ B\in\cJ$.  A {\em closed subideal of $L(X,Y)$} is a
 subideal which is closed in the operator norm.  
 We say that a subideal $\cJ\subset L(X,Y)$   is {\em non trivial} if it is not the {\em zero ideal } $\{0\}$ and proper if it is not all of $L(X,Y)$.

 The following is a list of some important   closed subideals of $L(X,Y)$.

 $\FD(X,Y)$ is the closure of the ideal of operators with finite dimensional rank. Note that 
  any nontrivial closed subideal $\cJ$  in $L(X,Y)$ contains all of $\FD(X,Y)$. This follows from the
   fact that $\cJ$ is closed under taking sums, under multiplication by elements
    of $L(X)$ from the right, under multiplication from the left by elements of $L(Y)$,  and that it must contain a non zero operator (and thus a rank 1 operator).
    Thus, for all infinite dimensional Banach spaces $X$ and $Y$ the ideal $\FD(X,Y)$ is the   minimal  nontrivial  closed subideal of $L(X,Y)$. 
    
   $\cK(X,Y)$ denotes the ideal of compact operators . All the spaces we consider are spaces
  with a basis.  Thus, these spaces  have the approximation property, which means  that 
  $\FD(X,Y)=\cK(X,Y)$. 
  
  $\StSi(X,Y)$ is the closed ideal of operators $T:X\to Y$ which are {\em strictly singular}, i.e. 
 on no infinite dimensional subspace $Z$ of $X$  is the restriction of  $T$ onto $Z$  an isomorphism.
 
  $\FSS$ is the closed ideal of {\em finitely strictly singular operators}. A linear bounded operator $T:X\to Y$, is called 
  {\em finitely strictly singular } if for all $\vp>0$ there is an $n=n_\vp\kin\N$ so that for any $n$-dimensional subspace $E$ of $X$,
  there is an $x\in E$, with $\|x\|=1$, so that $\|T(x)\|\le \vp$.
  
   If $W$ and $Z$ are Banach spaces and $S:W\to Z$ a bounded linear operator, we denote by $\cJ^S(X,Y)$ the closure of the ideal generated   by all
   operators  $T\in L(X,Y)$, which factor through $S$, thus $T= A\circ S \circ B$, with $A\in L(Z,Y)$  and $B\in L( X,W)$.
   In general the set $\{ A\circ S \circ B, A\in L(Z,Y) \text{ and } B\in L( X,W)\}$ is not closed under addition and therefore not an ideal.
    But if the operator 
    $$S\oplus S: W\oplus W\to Z\oplus Z, \qquad(w_1,w_2)\mapsto (S(w_1),S(w_2)),$$ 
    factors through $S$, then  $\{ A\circ S \circ B, A\in L(Z,Y) \text{ and } B\in L( X,W)\}$  is an ideal and we conclude
    in  that case that 
   \begin{equation}\label{E:1.2} \cJ^S(X,Y)=\overline{\{ A\!\circ\!S\!\circ\!B : A\in L(Z,Y) \text{ and } B\in L( X,W)\}}.
   \end{equation}
   Let $I(p,q):\ell_p\to \ell_q$  be the formal inclusion  (using that $\ell_p$ is a subset of $\ell_q$), for  $1\le p<q \le \infty$.
   It is easily seen that    $I(p,q)\oplus I(p,q)$ factors through $ I(p,q)$  and we conclude that
  $
   \cJ^{I(p,q)} (X,Y)=\overline{\{  A\!\circ\! I(p,q) \!\circ\!B: A\in L(\ell_q,Y) \text{ and } B\in L( X,\ell_p)\}}.$

   If $I_Z$ is the identity on some Banach space $Z$  we write $\cJ^Z$ instead of $\cJ^{I_Z}$, and we note that if $Z$ is isomorphic 
   to $Z\oplus Z$ it follows that
  \begin{equation}\label{E:1.2a}
   \cJ^Z(X,Y)=\overline{\{  A\!\circ\!S\!\circ\!B: A\in L(Z,Y) \text{ and } B\in L( X,Z)\}}.
   \end{equation}
   If $X=Y$ we will write $\cK(X)$, $\FSS(X)$ etc. instead of $\cK(X,X)$, $\FSS(X,X)$ etc.

For $1\le p<\infty$, we denote the unit vector basis of $\ell_p=\ell_p(\N)$ by $(e_{(p,j)}:j\kin\N)$ (if $p=\infty$ we consider $c_0$ instead of $\ell_\infty$). 
 The conjugate of $p$ is denoted by $p'$, i.e. $\frac1p+\frac1{p'}=1$.
For $n\kin\N$ we denote the $n$-dimensional 
$\ell_p$ space by $\ell_p(n)$ and its unit vector basis by  $(e_{(p,n,j)}:j\!=\!1,2,\ldots,n)$. The usual norm on $\ell_p$ or $\ell_p(n)$, $n\kin\N$ is denoted by $\|\cdot\|_p$.
 If $X_n$ is a Banach space for $n\kin\N$,
the {\em $\ell_p$-sum of $X_n$}, $n\kin\N$, is the space of all  sequences $(x_n:n\kin\N)$, with $x_n\in X_n$, for $n\in\N$,
and 
$$\|(x_n)_{n\in\N}\|_p=\Big(\sum_{n\in\N} \|x_n\|^p\Big)^{1/p}<\infty, \text{ if $p<\infty$, and }$$

We  denote the $\ell_p$-sum of $(X_n)$ by $(\oplus_{n=1}^\infty X_n)_{p}$. If $p= \infty $ we denote by 
$(\oplus_{n=1}^\infty X_n)_{\infty}$ the {\em $c_0$-sum,} the space of all sequences $(x_n)$, with $x_n\in X_n$, for $n\!\in\!\N$, for
which $\lim_{n\to\infty} \|x_n\|=0$.

The sphere and the unit ball of a Banach space are denoted by $S_X$ and $B_X$, respectively.
 For simplicity all our Banach spaces are defined over the real  field $\R$. It is easy to see how our results can be extended to Banach spaces
over the complex field $\C$.

\section{Review of the known results on the closed subideals of $L(\ell_p\oplus\ell_q)$ and $L(\ell_p,\ell_q)$}\label{S:2}

We will now review the known results on the lattice structure of subideals of  $L(\ell_p\oplus\ell_q)$.  
 We will assume from now on that $1<p<q<\infty$ and later  that $1\!<\!p\!<\!2\!<\!q\!<\!\infty$.

Every operator 
$T=\ell_p\oplus\ell_q \to  \ell_p\oplus\ell_q$, consists of four operators  $T_{(1,1)}\kin L(\ell_p)$, $T_{(1,2)} \in L(\ell_q,\ell_p)$ and $T_{(2,1)}\kin L(\ell_p,\ell_q)$,
and $T_{(2,2)}\in L(\ell_p,\ell_p)$, and acts as a  2 by 2 matrix on the elements of $\ell_p\oplus\ell_q$
$$T=\begin{pmatrix} T_{(1,1)} &T_{(1,2)}\\ T_{(2,1)} &T_{(2,2)}  \end{pmatrix}:\ell_p\!\oplus\!\ell_q\to \ell_p\!\oplus\!\ell_q,\quad
(x,y)\mapsto \big(T_{(1,1)}(x)\!+\!T_{(1,2)}(y),  T_{(2,1)}(x)\!+\!T_{(2,2)}(y)\big).$$

By the above cited  result from \cite{GMF},  the operators $T_{(1,1)}$ and $T_{(2,2)}$ are either compact or the identity on $\ell_p$, respectively $\ell_q$, factors through them.
 By Pitt's Theorem (c.f.  \cite[Proposition 6.25]{FHHMPZ}), $T_{(1,2)}$ is compact,  and since every infinite dimensional subspace of $\ell_p$ contains a subspace isomorphic to $\ell_p$, 
 and  since $\ell_p$ and $\ell_q$ are incomparable, we conclude that $T_{(2,1)}$ must be strictly singular.
So, if $\cJ$ is  a  closed subideal of $L(\ell_p\oplus \ell_q)$  which contains an operator $T$ for which $T_{(1,1)}$ and $T_{(2,2)}$ are  not 
compact, we conclude that the  identity on $\ell_p\oplus \ell_q$ factors through $T$ and thus $\cJ= L(\ell_p\oplus \ell_q)$. If $\cJ$ contains an operator for which
 $T_{(1,1)}$  is not compact, but for all elements  $U\in \cJ$, $U_{(2,2)}$ is compact, then  the identity on $\ell_p$ factors through $T$, but not the identity
  on     $\ell_q$, and  we therefore  deduce that $J$ must be the closure of  the operators factoring  through $\ell_p$, which must therefore be a maximal proper subideal 
   of $L(\ell_p\oplus \ell_q)$ (for more details see \cite[Theorem  5.3.2]{Pi}). Similarly we conclude that the closure of all operators factoring through $\ell_q$ is a maximal proper 
 subideal of $L(\ell_p\oplus \ell_q)$. 
 
For all other closed  proper subideals $\cJ\subset L(\ell_p\oplus \ell_q)$, and all $T\in \cJ$ it  therefore follows that  $T_{(1,1)}$, $T_{(1,2)}$ and $T_{(2,2)}$ are compact, and can therefore be 
 approximated by finite rank operators which factor through $\ell_p$ as well as $\ell_q$. Of course $T_{(2,1)}$  also factors through $\ell_p$ as well as $\ell_q$, and we deduce
 that all other closed ideals are subideals of  $\cJ^{\ell_p}(\ell_p\oplus\ell_q)\cap \cJ^{\ell_q}(\ell_p\oplus\ell_q)$, and thus not maximal proper closed ideals. 
 
 Assume now  that $\cJ\subset \cJ^{\ell_p}(\ell_p\oplus\ell_q)\cap \cJ^{\ell_p}(\ell_p\oplus\ell_q)$ is a closed ideal in $L(\ell_p\oplus\ell_q)$
    An easy computation yields that $\tilde\cJ:=\{T_{(2,1)}: T\in\cJ\}$ is a closed subideal of $L(\ell_p,\ell_q)$, and that for two  different 
   ideals 
  $\cJ_1,\cJ_2\subset \cJ^{\ell_p}(\ell_p\oplus\ell_q)\cap \cJ^{\ell_p}(\ell_p\oplus\ell_q)$ the ideals $\tilde\cJ_1$ and $\tilde\cJ_2$ are different.
   Conversely if   $\cJ$  is a closed  subideal  of $L(\ell_p,\ell_q)$ then
   $$\cJ'=\left\{ \begin{pmatrix} T_{(1,1)} &T_{(1,2)}\\ T_{(2,1)} &T_{(2,2)} \end{pmatrix} : T_{(2,10}\in\cJ\text{ and } T_{(1,1)}\in\cK(\ell_p),\, T_{(1,2)}\kin \cK(\ell_q,\ell_p), \text{ and } T_{(2,2)}\kin\cK(\ell_q)\right\}$$
   is a closed subideal of $L(\ell_p\oplus\ell_q)$ and for two different closed subideals $J_1,J_2\subset L(\ell_p,\ell_q)$, $\cJ_1'$ and $\cJ_2'$ are different.
  Thus there is a   bijection between the  set of all closed subideals of $L(\ell_p,\ell_q)$ and the non maximal closed subideals of  
    $L(\ell_p\oplus  \ell_q)$, which preserves the lattice structure  with respect to  inclusions.

Let us summarize  the observations we just made in the following proposition.
\begin{prop}  For $1<p<q<\infty$, the space   $L(\ell_p\oplus\ell_q)$ has exactly two maximal proper closed subideals, namely $\cJ^{\ell_p}(\ell_p\oplus\ell_q)$ and $\cJ^{\ell_q}(\ell_p\oplus\ell_q)$.

All other closed subideals of   $L(\ell_p\oplus\ell_q)$, are subideals of  
$\cJ^{\ell_p}(\ell_p\oplus\ell_q) \cap\cJ^{\ell_q}(\ell_p\oplus\ell_q)$, and there is 
 a bijection between the closed   subideals of  
$\cJ^{\ell_p}(\ell_p\oplus\ell_q)\cap\cJ^{\ell_q}(\ell_p\oplus\ell_q)$ and closed subideals of
 $L(\ell_p,\ell_q)$ which preserves the lattice structure.
\end{prop}

 We are therefore  interested in the closed subideals of $L(\ell_p,\ell_q)$. 
Instead of writing $\cK(\ell_p,\ell_q)$,  $\FSS(\ell_p,\ell_q)$, or $\cJ^S(X,Y)$  etc. we will from now on simply  write $\cK$,  $\FSS$ or $ \cJ^S$ etc.

The following diagram summarizes the  results established in \cite{Mi} and \cite{SSTT}, under the assumption that $1<p<2<q<\infty$.

\noindent
\ \ \ \ \ \ \ \ \ \ \ \ \ \ {\footnotesize
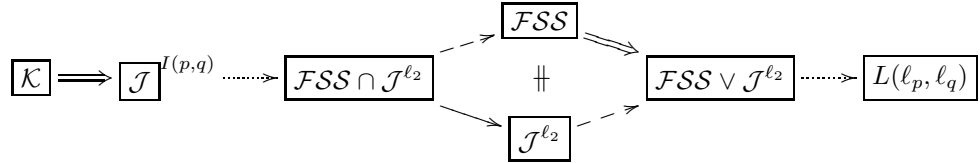
\xymatrix@R=0pt@C=20pt{    
& & &  \boxed{\FSS} \ar@{=>}[dr] & &\\
\boxed{\cK}\ar @{=>}[r]&\boxed{\cJ}^{I(p,q)} \ar@{.>}[r]  &\boxed{\FSS\cap\cJ^{\ell_2}} \ar@{-->}[ur] \ar[dr] &-\!\!\!\!\| &
  \boxed{\FSS\vee\cJ^{\ell_2}} \ar@{.>}[r] & \boxed{L(\ell_p,\ell_q)}\\
& & &  \boxed{\cJ^{\ell_2}  } \ar@{-->}[ur] & &
}}\\

Here arrows stand for inclusions. A solid arrow ($\Rightarrow$ or
$\to$) between two ideals means that there are no other ideals
sitting properly between the two, while a double arrow coming out of
an ideal indicates the only immediate successor.  A
hyphenated arrow ($--\!\!\!>$) indicates a proper inclusion, while a dotted
one indicates that we do not know whether or not the inclusion is
proper. In particular, the closed ideals in $L(\ell_p,\ell_q)$ are not totally
ordered.

Let us explain the diagram ``from the left to the right'' (for a more detailed explanation we refer the reader to \cite{SSTT}):

If $T: \ell_p\to\ell_q$ is not compact, then there is a  normalized  block sequence $(x_n)$ in $\ell_p$ whose image $(y_n)=(T(x_n)$ is equivalent to $(e_{(q,j)}:j\kin\N)$  (the   unit vector basis in $\ell_q$) and so that $\spa(y_n:n\kin\N)$ is complemented in $\ell_p$.
It follows that $I(p,q)$ factors through $T$, and that therefore $\cJ^{I(p,q)}$ is the only successor of $\cK$.     

It is clear that $\cJ^{I(p,q)}\subset \cJ^{\ell_2}$ (recall that we assume that $p<2<q$). The fact that   $\cJ^{I(p,q)}\subset \FSS$ follows from the following  result in \cite{Mi}
 (see also \cite[Proposition 3.3]{SSTT}).

\begin{prop}\label{P:1.1} For any choices of $1\le p<q\le \infty$ is the formal identity $I(p,q)$ is a finitely strictly singular operator.\end{prop}

The way to verify Proposition \ref{P:1.1}  is to  show  first (see \cite{Mi} or \cite[Lemma 3.4]{SSTT}) by induction on $n\kin\N$, that in every $n$-dimensional subspace $E$ of $c_0$
there is  $x\in E$ which attains its sup-norm on at least $n$ coordinates. In order to see then, that $I(p,q)$ is finitely strictly singular, let $\vp>0$ and pick $n\kin\N$ with
 $n^{-(q-p)/q}<\vp$. If $E$ is any subspace of $\ell_p$ of dimension $n$ we can find $x\in E$,  $\|x\|_p=1$,
  so that $\|x\|_\infty\le n^{-1/p} $ (since the maximum is attained on at least $n$ coordinates),
  and thus $\|x\|_q^q=\sum_{i=1}^\infty |x(i)|^{q-p}|x(i)|^p \le   \|x\|_\infty^{q-p} \|x\|_p^p \le n^{p-q}$ and thus $\|x\|_q\le n^{-(q-p)/q}  \le \vp$. We therefore established that
  $\cJ^{I(p,q)}\subset \FSS\cap\cJ^{\ell_2}$. In Section \ref{S:2} we will show that this inclusion is strict. 
  More precisely, we will show that the ideals $\cJ^{I(p,2)}$ and  $\cJ^{I(2,q)}$  are two distinct
  closed ideals which lie between $\cJ^{I(p,q)}$ and $ \FSS\cap\cJ^{\ell_2}$.

In order to show that $\FSS\cap\cJ^{\ell_2}$ is not all of $L(\ell_p,\ell_q)$  Milman \cite{Mi}  used the fact that  $\ell_p$  (and $\ell_q$)is   isomorphic the 
 $\ell_p$-sum (respectively the  $\ell_q$ sum) of $\ell_2(n)$, $n\in\N$  (see \cite[page 73]{LT}). Letting $U:\ell_p\to (\oplus_{n\in\N}\ell_2(n))_p$ and
  $V:\ell_q\to(\oplus_{n\in\N}\ell_2(n))_q$ be isomorphisms and  letting $I'(p,q)$ be  the formal identity  
  $$I'(p,q): (\oplus_{n\in\N}\ell_2(n))_p\to(\oplus_{n\in\N}\ell_2(n))_q,\qquad (x_n)\mapsto (x_n),$$
  we define $T{(p,q)}= V\circ I'(p,q) \circ U$.
   $T(p,q)$  depends on the choice of the isomorphisms $U$ and $V$, nevertheless it is easy to see that for any other isomorphisms  $\tilde U:\ell_p\to (\oplus_{n\in\N}\ell_2(n))_p$ and
  $\tilde V:\ell_q\to(\oplus_{n\in\N}\ell_2(n))_q$, the operator   $\tilde T{(p,q)}= \tilde V\circ I'(p,q) \circ \tilde U$, factors through $T(p,q)$ and vice versa, and thus
    $\cJ^{T{(p,q)}}=\cJ^{\tilde T(p,q)}$.
   Clearly $T{(p,q)}\not\in\FSS$, and thus $\FSS $ is a proper closed subideal of $L(\ell_p,\ell_q)$. 
 
 It is clear that $\cJ^{T{(p,q)}}\subset \cJ^{\ell_2}$. Conversely,   Theorem 4.7 in
   \cite{SSTT} shows that every operator $S:\ell_p\to\ell_q$, which factors through $\ell_2$, belongs to $\cJ^{T{(p,q)}}$, thus
   we deduce that   $\cJ^{T_{(p,q)}}= \cJ^{\ell_2}$. Moreover, if $S\in L(\ell_p,\ell_q)$  is not in $\FSS$, it follows from Khintchine's theorem (for more detail see
   Theorem \ref{T:2.2} in Section \ref{S:3} and the remarks thereafter) that for some $c>0$ there are $c$-complemented subspaces $F_n\subset \ell_p$, which  are $c$-isomorphic 
     to $\ell_2(n) $,  for $n\in\N$, on which $S$ is a $c$-isomorphism.  After   perturbing  $S$ we can find a sequence $(k_n)\subset\N$, so that  if we write
     $\ell_p$ as   an $\ell_p$-sum of   $\ell_p(k_n)$ and $\ell_q$ as the $\ell_q$-sum of $\ell_q(k_n)$, we can assume that $F_n\subset \ell_p(k_n)\subset \ell_p$ and 
     $S(F_n)\subset  \ell_q(k_n)\subset \ell_q$. 
    From this (see  \cite[Theorem 4.13]{SSTT}) it follows  that $T(p,q)$ factors through $S$.
    We deduce therefore that the ideal $\cJ^{\ell_2}\vee \FSS=\cJ^{T(p,q)}\vee\FSS$ (the closed ideal generated by the elements
     of $\FSS$ and $\cJ^{\ell_2}$)  is the only successor of $\FSS$.
    
   Finally we need to construct an operator $U:\ell_p\to\ell_q$ which is in $\FSS$ but cannot be approximated by operators  which factor through $\ell_2$. This will show that
   $\FSS$ and $\cJ^{\ell_2}$ are incomparable, they both strictly contain     $\FSS\cap\cJ^{\ell_2}$ and are properly contained in $\cJ^{\ell_2}\vee \FSS$.
   
   To do that we write $\ell_p$ as $\ell_p$ sum of $\ell_p(2^n)$, $n\in\N$, and $\ell_q$  as  $\ell_q$-sum of  $\ell_q(2^n)$, $n\in\N$.
   For $n\in\N\cup\{0\}$ let $H_n$ be the { $n$-th Hadamard matrix}. This is an $2^n$ by $2^n$ matrix with entries which are either $1$ or $-1$, and can be defined by induction
   as follows; $H_0=(1)$, and assuming that $H_n$ has been defined one puts $H_{n+1}=\begin{pmatrix} H_n &H_n\\ H_n &-H_n \end{pmatrix} $.
 
 It is easy to see that $H_n$ as operator from $\ell_1(2^n)\to   \ell_\infty(2^n)$ is of norm $1$, and that $2^{-n/2}H_n$ is a  unitary matrix (i.e., an isometry on $\ell_2(2^n)$). It follows therefore from the
 Riesz Thorin Interpolation Theorem  (c.f. \cite{BL}) that  $U_n=2^{-n\frac{1}{\min(p',q)}}H_n$ is of norm at most 1 as an operator in $L(\ell_p(2^n),\ell_q(2^n))$.
 
 We define 
 $$U:\ell_p=\big(\oplus_{n=1}^\infty \ell_p(2^n)\big)_p\to \big(\oplus_{n=1}^\infty \ell_p(2^n)\big)_q,\qquad (x_n)\mapsto (U_n(x_n)).$$
 The fact that $U$ can not be approximated by operators which factor though $\ell_2$ can be obtained from the following Corollary of 
 Theorem 9.13  in   \cite{DJT} (see also \cite[Theorem]{SSTT}). 
   
   \begin{prop}\label{P:1.2}  {\rm cf. \cite[Corollary]{SSTT} } 
  Let $m\in\N$, $C>1$, and $r>1$, and
  assume that
  $V$ is an invertible $m$ by $m$ matrix.
  Let $\delta=\|V^{-1}\|_{L(\ell_r',\ell_{r'})}$.
  Then 
  $\|B\|_{L(\ell_p,\ell_r)}\cdot\|A\|_{L(\ell_r,\ell_q)}\ge \delta^{-1}$
  for any factorization $V=AB$.
  Moreover, if $\widetilde V$ is another $m$ by $m$ matrix with 
  \begin{equation}\label{approx}
    \|{\widetilde V- V}\|_{L(\ell_p,\ell_q)}\le
    \bigl(2\max\limits_{1\le i\le m}\|{V^{-1}e_i}\|_p\bigr)^{-1},
  \end{equation}
  then it follows that for any factorization $\widetilde V=AB$ we have
  $\|B\|_{L(\ell_p,\ell_r)}\cdot\|A\|_{L(\ell_r,\ell_q)}\ge (2\delta)^{-1}$.
\end{prop}

If $q\not=p'$ then it is easy to see that $U$ is finitely strictly singular. Indeed if $ p'<q$, it follows that $U_n= 2^{-n/p'} H_n$, and we deduce  again form the Riesz Thorin Interpolation Theorem
that $U_n$ is as operator between $\ell_p(2^n)$ and $\ell_{p'}(2^n)$  of norm not larger than 1, and thus $U\in L(\ell_p,\ell_{p'})$. But this implies that 
$U$ (as element in $L(\ell_{p},\ell_q)$) factors through $I(p',q)$, which is finitely  strictly singular by Proposition \ref{P:1.2}. A similar argument shows that
if $p'>q$, and thus $p<q'$,  then $U$ factors through $I(p,q')$.

The hard case is the case $q=p'\not=2$, in which the previous factorization argument does not work.
In this case it is better to see $\ell_p(n)$ as the space  $L_p(n)$, the space of all  $p$- integrable functions on $\{1,2\ldots n\}$ with the normalized counting measure
(i.e. $\|x\|_{L_p}=\frac{1}{n^{/1p}} \|x\|_p)$). Using interpolation between Schatten $p$-classes one can prove the following result 
\begin{thm}\label{T:1.3}{\rm \cite[Theorem 6.5]{SSTT}}
  Suppose that $T\colon L_p(N)\to\ell_{p'}(N)$. Let $E$ be a $k$-dimensional subspace of $L_p(N)$, and
  $C_1$, $C_2$, and $C_3$ be positive constants such that
  \begin{enumerate}
  \item $\|T\|_{L(L_2(N),\ell_2(N))}\le 1$ and
        $\|T\|_{L(L_1(N),\ell_\infty(N)))}\le 1$;
  \item $E$ is $C_1$-isomorphic to $\ell_2^k$;
  \item $F=T(E)$ is $C_2$-complemented in $\ell_{p'}^N$; and
  \item $T_{|E}$ is invertible and $\big\|{(T_{|E})^{-1}}\big\|\le C_3$.
  \end{enumerate}
  Then $k\le\bigl(C_1^3C_2C_3^2K_G^2\bigr)^{p'}$. Here $K_G$ denotes the Grothendieck 
  constant.
\end{thm}

Now, if  $q=p'$, then we apply for $n\in\N$  Theorem \ref{T:1.3}  to $N=2^n$ and $T_n=\frac1{n^{1/p}} U_n=\frac1n H_n$ (note 
$T_n$ satisfies (1) of Theorem \ref{T:1.3} ).
  If $U$ where not finitely strictly singular, we could find constants  $C_1$, $C_2$ and $C_3$ and  for any $k\kin\N$ we would could find $n=n_k\kin\N$ large  enough so that 
(2) and (3) of Theorem \ref{T:1.3} are satisfied (using again Theorem \ref{T:2.2} in Section \ref{S:2}) for $T=T_n$. But this contradicts the conclusion of Theorem \ref{T:1.3}.

\section{Two  new closed ideals of  $L(\ell_p,\ell_q)$}\label{S:3}

We now state our main result, which exhibits two new closed subideals of $L(\ell_p,\ell_q)$,  and shows that $\cJ^{I(p,q)}\subsetneq \FSS\cap\cJ^{\ell_2}$ and increases the count of the known closed  proper and non trivial subideals 
 of $L(\ell_p,\ell_q)$ to 7.  
 \begin{thm}\label{T:2.1} Assume that $1<p<2<q<\infty$. Then  the two ideals $\cJ^{I(p,2)}$ and $\cJ^{I(2,q)}$ are two incomparable closed subideals of $\FSS\cap \cJ^{\ell_2}$.
\end{thm}
We assume from now on that $1<p<2<q<\infty$.
It is clear that 
 $\cJ^{I(p,q)} \subset \cJ^{I(p,2)}$ and that by Proposition \ref{P:1.1}  $\cJ^{I(p,2)}\subset \FSS \cap \cJ^\ell_2$ and similarly 
  $\cJ^{I(p,q)} \subset \cJ^{I(2,q)} \subset\FSS \cap \cJ^{\ell_2}$. We can therefore extend the diagram of Section \ref{S:2} to the following  diagram.

\medskip
\noindent
\ \ \ \ \  {\footnotesize
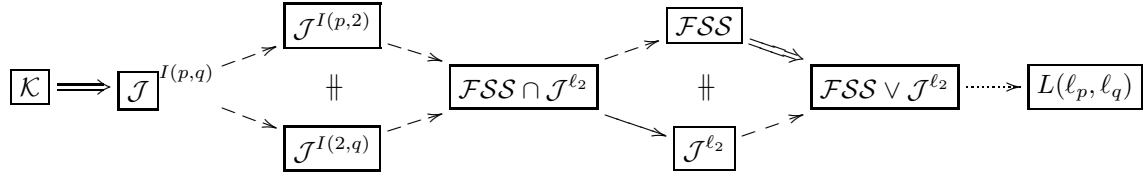
\xymatrix@R=0pt@C=20pt{    
& &   \boxed{\cJ^{I(p,2)}} \ar@{-->}[dr] & & \boxed{\FSS} \ar@{=>}[dr] & &\\
\boxed{\cK}\ar @{=>}[r]&\boxed{\cJ}^{I(p,q)}\ar@{-->}[ur] \ar@{-->}[dr] & -\!\!\!\!\|&\boxed{\FSS\cap\cJ^{\ell_2}} \ar@{-->}[ur] \ar[dr] &-\!\!\!\!\| &
  \boxed{\FSS\vee\cJ^{\ell_2}} \ar@{.>}[r] & \boxed{L(\ell_p,\ell_q)}\\
& &   \boxed{\cJ^{I(2,q)}}\ar@{-->}[ur]   &  &\boxed{\cJ^{\ell_2}  } \ar@{-->}[ur] & &
}}\\
This solves Question (i) in \cite{SSTT} and  shows that $\cJ^{I(p,q)}$ is different from $\FSS\cap \cJ^{\ell_2}$,
and that  the two (different) closed subideals  $\cJ^{I(p,2)}$  and   $\cJ^{I(2,q))}$  lie between them.

In order to show Theorem \ref{T:2.1} we need to find two operators $T$ and $S$  in $\FSS\cap \cJ^{\ell_2}$ , so that 
$T\in \cJ^{I(p,2)}\setminus \cJ^{I(2,q)}$ and $S\in  \cJ^{I(2,q)}\setminus \cJ^{I(p,2)}$.
We will first need the following  result. 
 \begin{thm}\label{T:2.2}
  For every $1<r<\infty$ there exists a  constant $K=K(r)>0$ and for all
  $n\in\mathbb N$ a number  $N=N(n,r)\in \N$,
  such that every $N$--dimensional subspace $F\subset\ell_r$
  contains an $n$--dimensional subspace $E$ which is $K$--complemented
in $\ell_r$ and $K$--isomorphic to~$\ell_2(n)$.
\end{thm}

\begin{remark}\label{R:1.2} Theorem \ref{T:2.2} follows from  the finite dimensional version of Khintchin's Theorem
(see \cite[Theorem 6.28] {FHHMPZ}).
  Better estimates on $N(n,r)$ and $K(r)$ can be obtained by applying
 simultaneously  Dvoretzky's theorem both
  to a subspace $F\subset\ell_r$ and to its dual $F^*$ (see e.g.,
  \cite{MS}). This gives the result with $N = C n^{r/2}$ and $K
  = C'\sqrt {\max\{r, r'\}}$, where $C, C' >0$ are absolute constants.   This theorem can also be viewed, for example, as a special case of
  results in~\cite{FT}.
\end{remark}

\begin{proof}[Proof of Theorem \ref{T:2.1}]
 We will now construct the  operators 
  $T\in \cJ^{I(p,2)}\setminus \cJ^{I(2,q)}$ and $S\in  \cJ^{I(2,q)}\setminus \cJ^{I(p,2)}$.

Put $C=\max (K(p),K(q))$ and for $n\in\N$ let $k_n=\max(N(p,n),N(q,n))$,
 where  $K(p)$, $K(q)$, $N(p,n)$  and $N(q,n)$ are chosen  as in Theorem \ref{T:2.2}.
 Using   that result  we can find for every $n\kin\N$ a
 sequence   $(x_{(n,i)})_{i=1}^n$ in  $C B_{\ell_p({k_n})}$
   so that
\begin{align}\label{E:2.1.1}
   &(x_{(n,i)})_{i=1}^n  \text{ is  $ C$-equivalent to  the unit vector basis of $\ell_2(n)$ and}   \\
   &\text{there is a projection $P_n$ from $\ell_p(k_n)$ onto $\spa( x_{(n,i)}:i=1,2,\ldots n)$ with $\|P_n\|\le C$.}\label{E:2.1.2}
 \end{align}
 For $n\in\N$ we define  $I_n:\spa( x^{(n)}_i: i=1,2\ldots, n)\to \ell_2(n)$,  by 
$I_n(x_{(n,i)})=e_{(2,n,i)}$, $i=1,\ldots n$. $I_n$ is thus a $C$-isomorphism. 
Writing $\ell_p$ as $\ell_p$-sum of $\ell_p(k_n)$and $\ell_2$ as $\ell_2$-sum of $\ell_2(n)$, $n\in\N$, we define $\tilde S$ as follows
 \begin{equation*}
\tilde S: \big(\oplus_{n=1}^\infty \ell_p(k_n)\big)_p\to \big(\oplus_{n=1}^\infty\ell_2(n)\big)_2,\quad (x_n)\mapsto  \big(I_n\circ P_n(x_n):n\kin\N\big).
 \end{equation*}
It follows that $\|\tilde S\|\le C^2$. Finally we let $S:=I(2,q)\circ \tilde S\in \cJ^{I(2,q)}$.

The construction of $T:\ell_p\to\ell_q$ is similar. Using again Theorem \ref{T:2.2} we find for each $n\kin\N$
vectors $(y_{(n,i)}:i=1,2\ldots n)$ in $CB_{\ell_q(k_n)}$ so  that 
\begin{align}\label{E:2.1.3}
   &(y_{(n,i)})_{i=1}^n  \text{ is  $ C$-equivalent to  the unit vector basis of $\ell_2(n)$, and}   \\
   &\text{there is a projection $Q_n$ from $\ell_q(k_n)$ onto $\spa( y_{(n,i)}: i=1,2,\ldots n )$ with $\|Q_n\|\le C$.}\label{E:2.1.4}
 \end{align}
Let $J_n:\ell_2(n) \to \ell_q(k_n)$, be the linear map which assigns to $e_{(2,n,i)}$ the vector  $y_{(n,i)}$, $i=1,2\ldots n$, then $J_n$ is a $C$-isomorphism onto
its image, and by writing again $\ell_2$ as $\ell_2$-sum of  $\ell_2(n)$ and $\ell_q$ as $\ell_q$-sum of $\ell_q(k_n)$, $n\in\N$, we define $\tilde T$ as
$$ \tilde T: \big(\oplus_{n=1}^\infty\ell_2(n)\big)_2\to \big(\oplus_{n=1}^\infty \ell_q(k_n)\big)_q,\quad (x_n)\mapsto \big(J_n(x_n):n\kin\N\big).$$
   Thus $\tilde T$ is a bounded operator with $\|\tilde T\|\le C$ and $T:= \tilde T\circ I(p,2)\in \cJ^{I(p,2)}$.
   
   In order to show that  $S\not\in  \cJ^{I(p,2)}$ and $T\not\in\cJ^{I(2,q)}$ we will find
   two functionals $\Phi$ and $\Psi$ in $L^*(\ell_p,\ell_q)$ so that
   $\Phi(S)=1$ and $\Phi|_{\cJ^{I(p,2)}}\equiv 0$, and, conversely 
   $\Psi(T)=1$ and  $\Psi|_{\cJ^{I(2,q)}}\equiv 0$ .

   Let $q'$ be the conjugate of $q$ (i.e. $\frac1q+\frac1{q'}=1$).    For $n\in\N$ we define 
 $$\tilde\Phi_n: L(\ell_p(k_n) ,\ell_q(n))\to \R,\quad \text{ with }\tilde\Phi_n(V)=\frac1n \sum_{i=1}^n  \la e_{(q',n,i)} ,V(x_{(n_i)})\ra.$$
Since by choice  $\| x_{(n,i)}\|\le C$, for $i=1,\ldots, n$,
  it  follows that $\|\tilde \Phi_n\|\le C$. We extend $\tilde\Phi_n$ in the canonical
  way to a functional in  $L^*(\ell_p ,\ell_q)$, i.e 
  let
   $E_n:\ell_p(k_n)\to \ell_p=(\oplus_{n=1}^\infty \ell_p(k_n))$ be the canonical embedding to the $n$-component and 
  let   $F_n  :\ell_q=(\oplus_{j=1}^\infty \ell_q(j))\to \ell_q(n)$ be  the projection onto the $n$-th component, for $n\in\N$ and put
     $\Phi_n(U)=\tilde\Phi_n(F_n\circ U\circ E_n)$ for $U\in L(\ell_p,\ell_q)$.
     Then also $\|\Phi_n\|\le C$  and we let $\Phi\in L^*(\ell_p,\ell_q)$ be a w$^*$ accumulation point of the sequence $(\Phi_n)$ in $L^*(\ell_p,\ell_q)$.
  Since $F_n \circ S \circ E_n(x_{(n,i)}) $ is the $i$-th unit vector in $\ell_q(n)$ it follows that $\Phi(S)\!=\!\lim_{n\to \infty} \Phi_n(S)\!=\!1$.

  The definition of $\Psi\in L^*(\ell_p,\ell_q)$ is as follows.  
   Since $(y_{(n,i)}:i=1,2,\ldots,n)$ is $C$-isomorphic to 
   $(e_{(2,n,i)}:i=1,2\ldots, n)$ and its  linear span is  $C$-complemented in $\ell_q(k_n)$,    we can find a sequence 
 $(y^*_{(n,i)}: i=1,2\ldots,  n) \subset \ell_{q'}(k_n) $,  which is $C$-isomorphic to 
   $(e_{(2,n,i)}:i=1,2\ldots, n)$, and satisfies
  $\la y^*_{(n,i)},y_{(n,j)}\ra =\delta_{(i,j)}$ for $1\le i,j\le n$.

  For $n\kin\N$  we can then write the projection $Q_n:\ell_q(k_n)\to \spa(y_{(n,i)}:i=1,2,\ldots,n)$  (which was introduced  in 
   \eqref{E:2.1.4}) as 
   $$Q_n=\sum_{i=1}^n    y_{(n,i)}\otimes y^*_{(n,i)}: \ell_q(k_n)\to \spa(y_{(n,i)} :i=1,2,\ldots, n),\quad z\mapsto \sum_{i=1}^n    y_{(n,i)} \la y^*_{(n,i)}, z \ra.$$  
   Then  we define  for $n\in\N$
   $$\tilde \Psi_n: L(\ell_p(n),\ell_q(k_n))\to \R\qquad\text{ by }\tilde\Psi(U)=\frac1n\sum_{i=1}^n   \la y^*_{(n,i)},U(e_{(p,n.i)})\ra.$$
   We let $\Psi_n$ be the canonical extension of $\tilde \Psi$ to a functional in $L^*(\ell_p,\ell_q)$, i.e.  for $U\in L(\ell_p,\ell_q)$ we let 
   $\Psi_n(U) =  \tilde\Psi(F'_n\circ U\circ E'_n)$, where $E_n':\ell_p(n)\to \ell_p=\big(\oplus_{j\in\N} \ell_p(j)  \big)_p$, is the canonical embedding into the $n$-th component,
    and $F_n':\big(\oplus_{j\in\N} \ell_q(k_j)  \big)_q\to \ell_q(k_n)$ is the projection onto the $n$-th component.
    Since $\|y^*_{(n,i)}\|_{q'}\le C$, for $i=1,2\ldots n$, it follows   that   $\|\Psi_n\|\le C$ and we let $\Psi\in L^*(\ell_p,\ell_q)$ be a w$^*$-accumulation point of  $(\Psi_n)$. Since $T(e_{(p,n,i)})=y_{(n,i)}$ for $i=1,2\ldots, n $,
    it follows that $\Psi(T)=\lim_{n\to\infty} \la \Psi_n,T\ra=1$.
  
  It is left to show that $\cJ^{I(p,2)}\subset \text{ker}(\Phi)$ and that  $\cJ^{I(2,q)}\subset \text{ker}(\Psi)$. To do so, we need  a result  which  is of independent interest  and  will
  therefore  be stated   separately and more generally than needed.    
\end{proof}

\begin{defn} Let $X$ be a  finite or infinite dimensional
 Banach space with a normalized  basis $(e_i)$. If $X$ is infinite dimensional put for
 $j\in\N$, 
 \begin{align*}
&n_X(j)=\min\Big\{\Big\|\sum_{i\in I} e_i\Big\|: I\subset\N, \#I=j\Big\}, \text{ and }N_X(j)=\max\Big\{\Big\|\sum_{i\in I} e_i\Big\|: I\subset\N, \#I=j\Big\},
\intertext{and if $j\le \dim(X)<\infty$, then put}
 &n_X(j)=\min\Big\{\Big\|\sum_{i\in I} e_i\Big\|: I\subset\{1,2,\ldots \dim(X)\}, \#I=j\Big\}
 \text{ and }\\
 &N_X(j)=\max\Big\{\Big\|\sum_{i\in I} e_i\Big\|: I\subset\{1,2,\ldots \dim(X)\}, \#I=j\Big\}.
\end{align*}
\end{defn}

\begin{lem}\label{L:2.5} Assume that $E$ and $F$ are two finite dimensional spaces, both 
having $C_u$-unconditional and normalized  bases $(e_i:i=1,2\ldots m)$ and $(f_j:j=1,\ldots n)$,
respectively.

Assume further that there are  $1<t<s<\infty$  and positive  constants $c_1$,  and $c_2$,  so that for all $\ell\in\N$
\begin{equation}\label{E:2.5.0}
N_E(\ell)\le c_1\ell^{1/s} \text{ and } n_F(\ell)\ge   c_2\ell^{1/t}.
\end{equation}
Then there exists a number  $c>0$, depending only on $s$, $t$, $c_u$, $c_1$,  and $c_2$, so that for 
every linear operator $T:E\to F$ and any $\rho>0$
\begin{equation}\label{E:2.5.0a}
\big|\big\{ i\le m: ||T(e_i)||_\infty=\max_{j\le n} |f^*_j(T(e_i))|\ge \|T\| \rho\big\}\big|\le  c  \rho^{\frac{-s^2}{(s-1)(s-t)}},
\end{equation}
where $(f_j^*)$ are the coordinate functionals to $(f_j)$.
Moreover, if $c_u=c_1=c_2=1$, then  we can choose $c=1$.
\end{lem}

\begin{cor}\label{C:2.6} Under the assumptions of Lemma \ref{L:2.5}, it follows that
\begin{equation}\label{E:2.6.1} 
\frac1m\sum_{i=1}^m \|T(e_i)\|_\infty \le \|T\|(1+c) m^{-r(s,t)},\text{ where $r(s,t)\!=\!\frac{(s-1)(s-t)}{(s-1)(s-t)+s^2}$, for $s\!>\!t\!\ge\!1$.}
\end{equation}
\end{cor}
\begin{proof}
First note that for any $\rho>0$ Lemma \ref{L:2.5}  yields
\begin{align*}\label{E:2.6.1}
\frac1m\sum_{i=1}^m \|T(e_i)\|_\infty &=\frac1m\sum_{i=1,   \|T(e_i)\|_\infty\le \rho\|T\|}^m \|T(e_i)\|_\infty+ \frac1m\sum_{i=1,   \|T(e_i)\|_\infty> \rho\|T\|}^m \|T(e_i)\|_\infty \\
 &\le  \|T\|\rho+  c \|T\|\frac{\rho^{\frac{-s^2}{(s-1)(s-t)}}}m.
\end{align*}
Then we let 
$$\rho=m^{-\frac{(s-1)(s-t)}{(s-1)(s-t)+s^2}},$$
which implies that 
\begin{align*}
\frac1m\sum_{i=1}^m \|T(e_i)\|_\infty &\le  \|T\|m^{-\frac{(s-1)(s-t)}{(s-1)(s-t)+s^2}}+ c\|T\| m^{-1}m^{\frac{(s-1)(s-t)}{(s-1)(s-t)+s^2} \frac{s^2}{(s-1)(s-t)}}\\
  &=  \|T\|m^{-\frac{(s-1)(s-t)}{(s-1)(s-t)+s^2}}+c\|T\|m^{-\frac{(s-1)(s-t)}{(s-1)(s-t)+s^2}}=(1+c)\|T\|m^{-r(s,t)}.
\end{align*}
\end{proof}

\begin{proof}[Proof of Lemma \ref{L:2.5}] For the sake of a better readability  we will assume that $c_1=c_2=c_u=1$. The general case follows in the same way. We can also assume that $\|T\|=1$.

Let $T: E\to F$ and write  $y_i=T(e_i)$ as  $y_i=\sum_{j=1}^n \beta(i,j) f_j$. 
 Let 
 $\rho>0$ and put 
 $$A=A_\rho=\big\{i\in\{1,2,\ldots,m\} : \max|\beta(i,j)|\ge \rho\big\}.$$ 
 For $i\in A$ choose $j_i\in\{1,2\ldots, n\}$ so that  $|\beta(i,j_i)|\ge \rho$.
 Let $\tilde A=\{j_i:i\in A\}$ and 
 for $j\in\tilde A$ let $A_j=\{i\in A: j_i=j\}$. In order to estimate $|A_j|$ and then $\tilde A$ we compute
 \begin{align*} 
|A_j|^{1/s}&\ge  N_E(|A_j|)    \quad\text{ (By \eqref{E:2.5.0})}\\
   &\ge \Big\|\sum_{i\in A_j}\sign(\beta(i,j))  e_j    \Big\|_E\\
                   &\ge \Big\|T\big(\sum_{i\in A_j} \sign(\beta(i,j)) e_j  \Big)  \Big\|_F \quad\text{ (Since $\|T\|=1$)}\\
                   &\ge \Big\la f^*_j, \sum_{i\in A_j} T\Big(\sum_{i\in A_j} \sign(\beta(i,j)) e_j  \Big) \Big\ra\\
                   &= \sum_{i\in A_j} |\beta(i,j)|
                   \ge |A_j|\rho \end{align*}
which yields  
$|A_j|^{1- \frac1s} \le \rho^{-1},$
and thus
$$|A_j|\le \rho^{-1/(1-\frac1s)}= \rho^{-\frac{s}{s-1}}.$$
Since $|A|=\sum_{j\in\tilde A} |A_j|\le |\tilde A|\cdot  \rho^{-\frac{s}{s-1}}$, we obtain
\begin{equation}\label{E:2.5.1}
|\tilde A|\ge |A|  \rho^{\frac{s}{s-1}}.
\end{equation}
Let $(r_j)_{j=1}^m$ be a Rademacher sequence on some probability space $(\Omega,\Sigma,\P)$,  this means that
$r_1,r_2,\ldots r_m$ are independent and  $\{\pm1\}$-valued, with $\P(\{r_j=1\})=\P(\{r_j=-1\})=1/2$  for $j=1,2\ldots n$.
We compute 
\begin{align*}
 |A|^{1/s}&\ge N_E(|A|)          \quad\text{ (By \eqref{E:2.5.0})}\\
            &\ge \E\Big(\Big\|\sum_{i\in A} r_i e_i \Big\|_E\Big)\\
             &\ge  \E\Big(\Big\|\sum_{i\in A} \sum_{j=1}^n r_i \beta(i,j)  f_j\Big\|_F\Big)\quad\text{ (Since $\|T\| \le 1$)}\\
            &=\E\Big(\Big\|\sum_{j=1}^n f_j \Big|\sum_{i\in A} r_i \beta(i,j)\Big| \Big\|_F\Big)\quad  \text{ (By $1$-unconditionality of $(f_j)$).} \\
           \intertext{ Applying the multidimensional version  of  Jensen's inequality (c.f 
              \cite[10.2.6, page 348]{Du})  to the convex function
             $\R^n\ni z\to \|\sum_{j=1}^n z_j f_j\|_F$ and the $\R^n$ valued random vector 
             $Z=\big(|\sum_{i\in A} r_i \beta(i,j)|:j\le n\big))$ we obtain }
          |A|^{1/s}     &\ge \Big\|\sum_{j=1}^n f_j \E\Big(\Big|\sum_{i\in A} r_i \beta(i,j)\Big| \Big)\Big\|_F   \\  
              &\ge \Big\|\sum_{j\in \tilde A} f_j \E\Big(\Big|\sum_{i\in A} r_i \beta(i,j)\Big| \Big)\Big\|_F \quad             
              \text{ (By 1-uncondtionality of $(f_j)$).}
\intertext{
 For each $j\in\tilde A$ there is an $i_j\in A$ so that $|\beta(i,j)|\ge \rho$. Let $r$ be
    anther  $\pm1$ random variable with $\P(r=1)=\P(r=-1)=1/2$, which is
   independent to $(r_j:j=1,\ldots m)$ then
  \begin{align*}\E\Big(\Big|\sum_{i\in A} r_i \beta(i,j)\Big| \Big)&=
   \E\Big(\Big| r_{i_j} \beta(i_j,j) +   r\sum_{i\in A\setminus\{i_j\}}  r_i \beta(i,j)\Big| \Big)\\
  &= \E\Bigg(\frac12\Big| r_{i_j} \beta(i_j,j) +\!\!  \sum_{i\in A\setminus\{i_j\}}  r_i \beta(i,j)\Big|
   + \frac12\Big| r_{i_j} \beta(i_j,j) - \!\!  \sum_{i\in A\setminus\{i_j\}} r_i \beta(i,j)\Big|\Bigg)\\
   &\ge  \E( | r_{i_j} \beta(i_j,j)|) \  \ge \rho\ quad \text{ (Since $|a+b| +|a-b|\ge 2|a|$}).
 \end{align*}  
   Using again the 1-unconditionality  of $(f_j:j=1,2\ldots n)$ we deduce therefore  that}
   |A^{1/s}|&\ge \Big\|\sum_{j\in \tilde A} f_j \E\Big(\Big|\sum_{i\in A} r_i \beta(i,j)\Big| \Big)\Big\|_F \ge \rho \Big\| \sum_{j\in \tilde A} f_j\Big\|_F \ge n_F(|\tilde A|),
   \end{align*}
   and thus by our  assumption
  \eqref{E:2.5.0} and by   \eqref{E:2.5.1} we obtain 
 $$ |A|^{1/s}\ge n_F(|\tilde A|)\ge |\tilde A|^{1/t} \ge  |A|^{1/t}  \rho^{\frac{s}{ts-t}} .$$
  Solving for $|A|$ yields
  $$|A|\le \rho^{- \frac{s}{ts-t}\frac{st}{s-t}}=\rho^{\frac{-s^2}{(s-1)(s-t)}}, $$
  which proves our claim.
\end{proof}
 \begin{proof}[Continuation of Proof of Theorem \ref{T:2.1}]   
 In order to show that $\cJ^{I(p,2)}\subset \text{ker}(\Phi)$ we let $A\in L_2(\ell_2, \ell_q)$  and $B\in L(\ell_p)$. We need to show that
 $\Phi(A\circ I(p,2)\circ B)=0$. W.lo.g. we assume that $\|A\|,\|B\|\le 1$.
 
 Consider $B'_n:\ell_2(n)\to \ell_p(k_n)$ with $B'(e_{(2,n,i)})=B(x_{(n,i)})$, where we consider $\ell_p(k_n)$ canonically embedded  into $\ell_p=\big(\oplus_{j=1}^\infty \ell(k_j) \big)$. Then $\|B_n'\|\le C$ and applying therefore Corollary \ref{C:2.6} to $B'$, $s=2$ and $t=p$, we obtain   
 $$
 \frac1n\sum_{i=1}^n \|B(x_{(n,i)})\|_\infty = \frac1n\sum_{i=1}^n \|B_n'(e_{(2,n,i)})\|_\infty  \le 2Cn^{-r(2,p)}.
$$
which by the concavity of  $[0,\infty)\ni\xi\mapsto \xi^{(2-p)/2}$ implies that
 \begin{equation}\label{E:2.1.7}
 \frac1n\sum_{i=1}^n \|B(x_{(n,i)})\|_\infty^{(2-p)/2} \le 
 \Big( \frac1n\sum_{i=1}^n \|B(x_{(n,i)})\|_\infty \Big)^{(2-p)/2}
  \le (2C)^{(2-p)/2}n^{-r(2,p)(2-p)/2}.
 \end{equation}
 Secondly we observe that for any $i=1,2\ldots n$
\begin{align}\label{E:2.1.8}
\|I(p,2)(B(x_{(n,i)}))\|_2&=\Big(\sum_{j=1}^{k_n} |B(x_{(n,i)})(j)|^2\Big)^{1/2} \\
 &=\Big(\sum_{j=1}^{k_n} |B(x_{(n,i)})(j)|^{p} |B(x_{(n,i)})(j)|^{2-p}\Big)^{1/2} \notag\\
 &\le  \|B(x_{(n,i)})\|_{\infty}^{(2-p)/2} \cdot  \|B(x_{(n,i)})\|_p^{p/2}\le  C^{p/2} \| B(x_{(n,i)})\|_{\infty}^{(2-p)/2}. \notag
 \end{align}
 It follows therefore that
 \begin{align*}
\big| \Phi_n(A\circ I(p,2)\circ B)\Big|&= \frac1n\Big| \sum_{i=1}^n \la e_{(q',n,i)}, A\circ I_{(p,2)}\circ B(x_{(n,i)})   \ra \Big|\\
                                              &= \frac1n\Big| \sum_{i=1}^n  \la A^*(e_{(q',n,i)}), I_{(p,2)}\circ B(x_{(n,i)})   \ra \Big|\\
                                              &\le  \frac1n\sum_{i=1}^n  \| A^*(e_{(q',n,i)})\|_2 \| I_{(p,2)}\circ B(x_{(n,i)}\|_{2}  \\
                                              &\le \|A^*\| C^{p/2} \frac1n \sum_{i=1}^n  \| x_{(n,i)}\|_{\infty}^{(2-p)/2}\quad \text{ (By \eqref{E:2.1.8})}  \\ 
                                               &\le C^{p/2}(2C)^{(2-p)/2} n^{-r(2,p)(2-p)/2}\to_{n\to\infty} 0\quad \text{  (By \eqref{E:2.1.7})}.
 \end{align*}
 This implies that $\cJ^{I(p,2)}\subset \text{ker}(\Phi)$.
 
  In order to show that $\cJ^{I(2,q)}\subset \text{ker}(\Psi)$,  let $B\in L(\ell_p,\ell_2)$ and $A\in L(\ell_q)$
  with $\|B\|,\|A\|\le 1$. We need to show that
  $\Psi(A\circ I_{(2,q)}\circ B)=0$.

Let $A'_n:\ell_2(n)\to \ell_{q'}(k_n)$, defined by $A_n'(e^{(2,n,i)})=A^*(y^*_{(n,i)})$, $i=1,2\ldots n$
 (we consider $\ell_{q'}(k_n)$ in the canonical way  as subspace 
  of $\ell_{q'}=(\oplus_{j=1}^\infty\ell_{q'}(k_n))_q$).
 It follows from the choice of $(y^*_{(n,i)}:i=1,2\ldots n)$ that 
 $\|A'_n\|\le C$ and from Corollary \ref{C:2.6}  (with $s=2$ and $t=q'$) we deduce  that
 $$ \frac1n\sum_{i=1}^n \|A^*(y^*_{(n,i)})\|_{\infty} 
 = \frac1n\sum_{i=1}^n \|A'(e_{(2,n,i)})\|_{\infty} 
 \le 2Cn^{-r(2,q')}.$$
 Using the concavity of the function $[0,\infty)\ni\xi\to \xi^{(2-q')/2}$ we deduce
 \begin{equation}\label{E:2.1.9a}
  \frac1n\sum_{i=1}^n \|A^*(y^*_{(n,i)})\|_\infty^{(2-q')/2}=    
    \Big(\frac1n\sum_{i=1}^n \|A^*(y^*_{(n,i)})\|_\infty\Big)^{(2-q')/2}
    \le  (2C)^{(2-q')/2}n^{-r(2,q')(2-q')/2}.
 \end{equation}
 It is easy to see that $I_{(q',2)}$ is  the adjoint of $I_{(2,q)}$ and we compute for $i=1,2\ldots n$
\begin{align}\label{E:2.1.9}
\|I(q',2)\circ A^*(y^*(n,i))\|_2&=
 \Big(\sum_{j=1}^{k_n} \big(A^*(y^*(n,i))(j)\big)^2\Big)^{1/2}\\
&= \Big(\sum_{j=1}^{k_n} \big|A^*(y^*(n,i))(j)\big|^{q'}   \big|A^*(y^*(n,i))(j)\big|^{2-q'}\Big)^{1/2} \notag\\
 &\le  \|A^*(y^*{(n,i)})\|_\infty^{(2-q')/2} \|y_{(n,i)}\|_{q'}^{q'/2}  \notag \\
 &\le C^{q'/2}\|A^*(y^*{(n,i)})\|_\infty^{(2-q')/2} .\notag
 \end{align}
 Therefore it follows 
\begin{align*}
|\la \psi_n,  U\ra|
 &=\frac1n\Big|\sum_{i=1}^n \la  A\circ I_{(2,q)}\circ B(e_{(p,n,i)}), y^*{(n,i)}\ra\Big|\\
&=\frac1n\Big|\sum_{i=1}^n \la   B( e_{(p,n,i)}), I_{(q',2)}\circ A^* (y^*{(n,i)}  ) \ra\Big|\\
&\le \frac1n\sum_{i=1}^n \|  B( e_{(p,n,i)})\|_2\cdot \|  I_{(q',2)}\circ A^*( y^*_{(n,i)})  \|_2\\
&\le \|B\| C^{q'/2}\frac1n\sum_{i=1}^n
\|A^*(y^*{(n,i)})\|_\infty^{(2-q')/2}
 \text{\ \  (By \eqref{E:2.1.9})}\\
&\le  C^{q'/2}  (C+1)^{(2-q')/2}n^{-r(2,q')(2-q')/2}\to_{n\to\infty} 0
\text{\ \ (By \eqref{E:2.1.9a})}.
\end{align*}
which implies our claim, and finishes the proof or Theorem \ref{T:2.1}.
 \end{proof}

\end{document}